%% file: main.tex
\documentclass[11pt]{article}
\usepackage[a4paper,margin=2.6cm]{geometry}
\usepackage{amsmath,amssymb,amsthm,mathtools}
\usepackage[round,authoryear]{natbib}
\usepackage{graphicx,booktabs,xcolor,microtype}
\usepackage{placeins}
\usepackage[font=small,labelfont=bf]{caption}
\usepackage[colorlinks=true,citecolor=blue!50!black,
  linkcolor=blue!50!black,urlcolor=blue!50!black]{hyperref}
\usepackage[authoryear,round]{natbib}

\newif\ifshowtodo \showtodotrue

\makeatletter
\def\@fnsymbol#1{\ensuremath{\ifcase#1\or*\or\dagger\or
  \ddagger\fi}}
\makeatother

\newtheorem{theorem}{Theorem}
\newtheorem{proposition}{Proposition}
\newtheorem{lemma}{Lemma}
\newtheorem{corollary}{Corollary}
\theoremstyle{definition}
\newtheorem{remark}{Remark}
\newtheorem{openproblem}{Open problem}

\newcommand{\R}{\mathbb{R}}
\newcommand{\E}{\mathrm{E}}
\newcommand{\JB}{\mathrm{JB}}
\newcommand{\Disc}{\operatorname{Disc}}
\newcommand{\sph}{\Omega_n}
\newcommand{\EK}{\boldsymbol{K}}
\newcommand{\EF}{\boldsymbol{F}}

\title{Exact finite-sample geometry of the Jarque--Bera statistic:\\ support, discriminants and elliptic integrals} \author{Steve Lawford\thanks{ENAC (University of Toulouse), 7 avenue Edouard Belin, BP 54005, 31055 Toulouse Cedex 4, France (email: steve.lawford@enac.fr).}}

\date{}

\begin{document}
\maketitle

\begin{abstract}
\noindent Under Gaussian sampling, removing location and scale
turns the sample into a direction that is uniformly
distributed on a sphere of dimension $n-2$, and sample
skewness and kurtosis become a cubic--quartic polynomial image
of spherical measure. We use this geometry to obtain exact
finite-sample results for the Jarque--Bera statistic $\JB_n$.
First, for every $n\ge 3$ the largest attainable value of
$\JB_n$ is $n\{(n-2)^4+4(n-1)^2\}/\{24(n-1)^2\}$, attained by
a single-outlier sample, so the asymptotic $10\%$ and $5\%$
tests have size zero for $n\le 6$ and the $1\%$ test for
$n\le 7$; near this maximum the density of $\JB_n$ behaves
like an explicit multiple of $(J_n^+-x)^{(n-4)/2}$. Second,
passing from residual coordinates to power sums writes the
joint density of skewness and kurtosis as an integral of the
reciprocal square root of a polynomial discriminant; for every
$n\ge 5$ the innermost integral runs over a single interval
and is a Lauricella $F_D$ period. Third, this yields explicit
laws: an arcsine law for $n=3$, an algebraic joint density for
$n=4$, and a single complete elliptic integral, equivalently
${}_2F_1(\tfrac12,\tfrac12;1;\cdot)$, for $n=5$.
Residual-coordinate collisions generate the discriminant
singularities of the joint law, while stationary points of
$\JB_n$ on the residual sphere govern the singularities of its
one-dimensional density.
\end{abstract}


\section{Introduction}\label{sec:intro}

For observations $X_1,\dots,X_n$ let $m_r=n^{-1}\sum_{i=1}^n(X_i-\bar X)^r$, $S_n=m_3/m_2^{3/2}$ and $K_n=m_4/m_2^2$. The Jarque--Bera statistic \citep{JB1980,JB1987},
\begin{equation}\label{eq:JB}
\JB_n=\frac n6\Bigl\{S_n^2+\tfrac14(K_n-3)^2\Bigr\},
\end{equation}
is one of the canonical moment-based diagnostics for Gaussianity, and under independent normal sampling it converges in distribution to $\chi^2_2$. Its finite-sample law has a geometry that is absent from the limiting distribution. For every fixed $n$, its support is bounded, and at the smallest sample sizes the usual asymptotic critical values lie beyond the attainable range of the statistic, so the corresponding asymptotic test has exact size zero (Theorem~\ref{thm:range}, Corollary~\ref{cor:size}).

The mechanism is simple. After location and scale are
removed, a Gaussian sample is a uniformly distributed
direction on a residual sphere, so $(S_n,K_n)$ is the image
of spherical measure under a cubic--quartic polynomial map
(Section~\ref{sec:sphere}). Passing from residual coordinates
to power sums introduces the Vandermonde determinant, and the
joint density of $(S_n,K_n)$ becomes an integral of the
reciprocal square root of a polynomial discriminant
(Theorem~\ref{thm:disc}).

This representation gives exact finite-sample results and
exposes a hierarchy of algebraic and special-function
structure. The innermost integral is a Lauricella period for
every $n\ge5$ (Theorem~\ref{thm:fibre}); the joint density is
algebraic for $n=4$ (Theorem~\ref{thm:n4}) and a single
complete elliptic integral for $n=5$ (Theorem~\ref{thm:n5});
and the singularities have two geometric sources: collisions
among residual coordinates determine the discriminant
singularities of the joint law, while stationary points of
$\JB_n$ on the residual sphere determine the singularities of
the scalar $\JB_n$ density (Propositions~\ref{prop:tail}
and~\ref{prop:JB4sing}).
Section~\ref{sec:JB} converts these results into distribution
functions and quantiles of $\JB_n$, and
Section~\ref{sec:open} discusses larger $n$.

Exact theory has a long history. \citet{Fisher1930}
introduced recurrences in $n$ for the standardised sample
cumulants, obtained exact moments of their joint
distribution, found the density of $S_n$ for $n=3$, and used
the spherical representation, as did \citet{Mulholland1977}.
\citet{McKay1933} showed that for $n=4$ the density of $S_n$
is a complete elliptic integral, which he expressed as a Gauss
hypergeometric function, and \citet{Geary1947} developed the
recurrence for the density of $S_n$. \citet{Mulholland1970}
gave a general theory of the singularities of the density of
a function of a point uniformly distributed on a sphere,
which he applied to $S_n$ \citep{Mulholland1977}. For $K_n$,
\citet{AnscombeGlynn1983} gave an accurate approximation, and
\citet{ShentonBowman1977} an approximate bivariate model,
combining a Johnson $S_U$ marginal for $S_n$ with a gamma
conditional density for $K_n$ above the parabola $K_n=1+S_n^2$.
The main exact result for the joint law is due to
\citet{NHN2016}, hereafter NHN, who derive recurrences in $n$
for the joint density $h_n(s,k)$ of $(S_n,K_n)$ and for all
product moments, and hence the exact moments of $\JB_n$;
NHN leave the numerical evaluation of the density recurrence,
and hence graphical display of the joint law, for future work,
and \citet{NHO2021} have since approximated the density of
$S_n$ by moment-based Fourier cosine series. Critical values
of $\JB_n$ have been tabulated or fitted by response surfaces
\citep{DebSefton1996,Lawford2005}, and \citet{Urzua1996}
proposed a finite-sample restandardisation. A separate
literature bounds $S_n$ and $K_n$ algebraically in terms of
$n$ \citep{Wilkins1944,Kirby1974,JohnsonLowe1979,Dalen1987,
SharmaBhandari2015}; \citet{Cox2010} notes that these limits
have been repeatedly rediscovered. For $n=4$,
\citet{DeMicheleDeBartolo2026} observed in simulations that
the admissible $(S_4,K_4)$ region is deltoid-shaped, and
\citet{DeBartolo2026} have derived its boundary analytically.

Relative to these results, the contributions of the present
paper are the closed-form range of $\JB_n$ for all $n$
and its consequence for the size of the asymptotic test,
neither of which we have found stated previously; the
discriminant and Lauricella structure, in which the $n=4$
boundary appears as the discriminant locus of a quartic that
also carries the exact density on its interior; and complete
explicit low-dimensional laws together with their singularity
structure, which neither the recurrence nor the bounds
literature displays. NHN already give an exact recursive
characterisation of the joint law, and \citet{DeBartolo2026}
independently derive the $n=4$ boundary.

\section{Spherical representation and the range of
\texorpdfstring{$\JB_n$}{JBn}}\label{sec:sphere}

Let $X=(X_1,\dots,X_n)'$ have independent $N(\mu,\sigma^2)$
entries, $M=I_n-n^{-1}\mathbf 1\mathbf 1'$ and $Y=MX$. Put
$H_n=\{u\in\R^n:\mathbf 1'u=0\}$ and
$\sph=\{u\in H_n:\|u\|=1\}$, a sphere of dimension $n-2$. For
$u\in\R^n$ write $p_r(u)=\sum_{i=1}^n u_i^r$, so that
$p_1=0$ and $p_2=1$ on $\sph$.

\begin{lemma}\label{lem:sphere}
$Y\neq 0$ almost surely, $U=Y/\|Y\|$ is uniformly distributed
on $\sph$, and $U$ is independent of $\|Y\|$.
\end{lemma}

\begin{proof}
$Y$ is Gaussian with covariance $\sigma^2M$, and $M$ is the
orthogonal projection onto $H_n$, so the law of $Y$ is
isotropic on the $(n-1)$-dimensional space $H_n$. The polar
decomposition of an isotropic Gaussian vector gives the
result.
\end{proof}

\begin{proposition}\label{prop:power}
Under normality, $S_n=\sqrt n\,p_3(U)$, $K_n=n\,p_4(U)$ and
\begin{equation}\label{eq:JBp}
\JB_n=\frac{n^2}{6}\,p_3^2+\frac{n}{24}\,(np_4-3)^2 .
\end{equation}
\end{proposition}

\begin{proof}
With $Y_i=\|Y\|U_i$ one has $m_r=\|Y\|^r p_r(U)/n$ and
$m_2=\|Y\|^2/n$; the radial factors cancel.
\end{proof}

Thus the finite-sample problem is the push-forward of uniform
measure on $\sph$ under $u\mapsto(p_3,p_4)$. The following
bounds are classical.

\begin{lemma}\label{lem:bounds}
For every $u\in\sph$, with $S=\sqrt n\,p_3(u)$ and
$K=n\,p_4(u)$,
\begin{equation}\label{eq:bounds}
K\ge 1+S^2,\qquad K\le B_n=\frac{n^2-3n+3}{n-1},\qquad
|S|\le A_n=\frac{n-2}{\sqrt{n-1}} .
\end{equation}
Moreover $A_n^2=B_n-1$, and all three bounds are attained
simultaneously at the single-outlier directions
$u=\pm\{n(n-1)\}^{-1/2}(n-1,-1,\dots,-1)'$ and their
coordinate permutations.
\end{lemma}

\begin{proof}
As in \citet{JohnsonLowe1979}, let $Z$ take the values
$\sqrt n\,u_i$ with probability $1/n$ each, so that $\E Z=0$,
$\E Z^2=1$, $\E Z^3=S$, $\E Z^4=K$.
For every real $a$,
$0\le\E(Z^2-aZ-1)^2=K-1-2aS+a^2$; taking $a=S$ gives the
first inequality. It was stated by \citet{Pearson1916} and
proved by \citet{Wilkins1944}, who noted that equality holds
if and only if the data take at most two distinct values.
The bound on $|S|$ was proved by \citet{Wilkins1944} and,
independently, by \citet{Kirby1974}. The bound on $K$ was
stated by \citet{Pearson1916}, who attributed it to
G.\,N.~Watson, and proved by \citet{Dalen1987}, who also
showed that it is attained only at the single-outlier
configurations. Cruder bounds were given by
\citet[p.~357]{Cramer1946} and \citet{JohnsonLowe1979};
\citet{Cox2010} surveys the history. The identity
$A_n^2=B_n-1$ is immediate, and direct substitution of the
single-outlier direction gives $S=\pm A_n$ and $K=B_n$.
\end{proof}

\citet{SharmaBhandari2015} also give the $S$-dependent upper
bound $K\le\frac12\frac{n-3}{n-2}S^2+\frac n2$. For $n=5$
this bound is used below to identify the lower endpoint of
$\JB_5$; for $n=4$ it is compared with the exact support in
Remark~\ref{rem:triangle}.

\begin{theorem}[Range of $\JB_n$]\label{thm:range}
For $n\ge 3$, let $J_n^-$ and $J_n^+$ be the minimum and
maximum of $\JB_n$ over $\sph$. Then
\begin{equation}\label{eq:Jplus}
J_n^+=\frac{n}{24}\bigl(A_n^4+4\bigr)
=\frac{n\{(n-2)^4+4(n-1)^2\}}{24(n-1)^2},
\end{equation}
attained at the single-outlier directions of
Lemma~\ref{lem:bounds}. Further, $J_3^-=9/32$,
$J_4^-=1/6$, $J_5^-=5/96$, and $J_n^-=0$ for $n\ge 6$.
\end{theorem}

\begin{proof}
By Lemma~\ref{lem:bounds}, $S^2\le K-1$, so
\[
\JB_n\le\frac n6\Bigl\{K-1+\tfrac14(K-3)^2\Bigr\}
=\frac n{24}\bigl\{(K-1)^2+4\bigr\}
\le\frac n{24}\bigl\{(B_n-1)^2+4\bigr\},
\]
because the middle expression increases in $K\ge 1$. Since
$B_n-1=A_n^2$, the right-hand side is \eqref{eq:Jplus}, and
both inequalities are equalities at the single-outlier
directions. The values of $J_3^-$ and $J_4^-$ are proved in
Proposition~\ref{prop:n3} and Proposition~\ref{prop:JB4}.
For $n\ge 6$ let $m=\lfloor n/2\rfloor$ and consider the
directions with $u_{2i}=-u_{2i-1}$ for $i\le m$ (and $u_n=0$
if $n$ is odd). They form a sphere of dimension $m-1\ge 2$,
hence a connected set, on which $S=0$. On it $K$ is
continuous, equals $n/(2m)\le 7/6$ when all $|u_i|$ with
$i\le 2m$ are equal, and equals $n/2\ge 3$ when a single pair
is non-zero. By the intermediate value theorem $K=3$, and so
$\JB_n=0$, is attained.

For $n=5$, the bound of \citet{SharmaBhandari2015} reads
$K\le\frac52+\frac13S^2$. Put $t=S^2$. If $0\le t\le\frac32$,
then $K\le3$, so $3-K\ge\frac12-\frac t3\ge0$ and
\[
\JB_5=\frac56\Bigl\{t+\tfrac14(K-3)^2\Bigr\}
\ge\frac56\Bigl\{t+\tfrac14\Bigl(\tfrac12-\tfrac t3\Bigr)^2\Bigr\}
=\frac56\Bigl(\tfrac1{16}+\tfrac{11}{12}\,t+\tfrac1{36}\,t^2\Bigr)
\ge\frac5{96}.
\]
If $t\ge\frac32$, then $\JB_5\ge\frac56t\ge\frac54$. Equality
holds at $u=2^{-1/2}(1,0,0,0,-1)'$, where $(S,K)=(0,\frac52)$.
Hence $J_5^-=5/96$.
\end{proof}

\begin{corollary}[Exact size zero]\label{cor:size}
Let $c_\alpha=-2\log\alpha$ be the upper $\alpha$ point of
$\chi^2_2$. The test rejecting when $\JB_n>c_\alpha$ has
exact size zero if and only if $J_n^+\le c_\alpha$. Hence
the asymptotic $10\%$ and $5\%$ tests have exact size zero if
and only if $n\le 6$, and the $1\%$ test if and only if
$n\le 7$.
\end{corollary}

\begin{proof}
If $J_n^+>c_\alpha$, then $\{u:\JB_n(u)>c_\alpha\}$ is a
non-empty open subset of $\sph$ and has positive uniform
measure. By \eqref{eq:Jplus}, $J_n^+$ increases in $n$, since
both $n$ and $A_n^2=(n-2)^2/(n-1)$ do. Finally
$J_6^+=89/25=3.56<c_{0.10}=4.605$,
$J_7^+=5383/864\approx 6.230$ lies in $(c_{0.05},c_{0.01})=
(5.991,9.210)$, and $J_8^+=1492/147\approx 10.150>c_{0.01}$.
\end{proof}

Although Theorem~\ref{thm:range} follows in a few lines from
the classical bounds, we have not found $J_n^+$ or
Corollary~\ref{cor:size} stated in the literature on sample
moment bounds or on the Jarque--Bera test.
At $n=7$ the size is positive but negligible: in $10^8$
replications the rejection frequencies at the $10\%$ and
$5\%$ asymptotic critical values are $0.0018$ and
$1.2\times10^{-5}$. The next result gives the exact leading
behaviour of the upper tail of $\JB_n$. Its final step is an
instance of the theory of \citet{Mulholland1970,
Mulholland1977} for densities of functions on a sphere at a
nondegenerate stationary point; what is specific to $\JB_n$
is the identification of the maximisers and of their
Hessian.

\begin{proposition}[Upper tail]\label{prop:tail}
Let $n\ge3$ and
\[
\lambda_n=\frac{n^2(n^3-6n^2+15n-12)}{3(n-1)^2}.
\]
The maximum $J_n^+$ is attained exactly at the $2n$
single-outlier directions, and at each of them the Hessian
of $\JB_n$ on $\sph$ is $-\lambda_n$ times the identity.
Consequently, with $d=n-2$ the dimension of $\sph$, as
$\varepsilon\downarrow0$,
\begin{equation}\label{eq:tail}
P(\JB_n>J_n^+-\varepsilon)=c_n\Bigl(\frac{2\varepsilon}
{\lambda_n}\Bigr)^{d/2}\{1+o(1)\},\qquad
c_n=\frac{n\,\Gamma\{(n-1)/2\}}{\pi^{1/2}\,\Gamma(n/2)},
\end{equation}
and, as $x\uparrow J_n^+$,
\begin{equation}\label{eq:tailden}
f_{\JB_n}(x)=\frac{n\,\Gamma\{(n-1)/2\}}{\pi^{1/2}\,
\Gamma\{(n-2)/2\}}\Bigl(\frac{2}{\lambda_n}\Bigr)^{d/2}
(J_n^+-x)^{d/2-1}\{1+o(1)\}.
\end{equation}
\end{proposition}

\begin{proof}
\emph{Maximisers.} Equality in the proof of
Theorem~\ref{thm:range} requires $S^2=K-1$ and $K=B_n$. By
the proof of Lemma~\ref{lem:bounds}, $S^2=K-1$ if and only if
$\E(Z^2-SZ-1)^2=0$, that is, every $\sqrt n\,u_i$ is a root of
$z^2-Sz-1$, so the coordinates take at most two values
\citep[cf.][]{Wilkins1944}. If a
fraction $q$ of them take one value, $Z$ is a standardised
two-point variable and $K=(1-3q+3q^2)/\{q(1-q)\}$, which
decreases in $q(1-q)$; hence $K=B_n$ if and only if
$q\in\{1/n,1-1/n\}$, in agreement with \citet{Dalen1987}.
The maximisers are therefore the $n$ coordinate permutations
of $\pm\{n(n-1)\}^{-1/2}(n-1,-1,\dots,-1)'$.

\emph{Hessian.} Since $\JB_n(-u)=\JB_n(u)$ and $\JB_n$ is
invariant under coordinate permutations, it suffices to take
$u^*=\{n(n-1)\}^{-1/2}(n-1,-1,\dots,-1)'$. A tangent vector
$v$ at $u^*$ satisfies $\mathbf 1'v=0$ and $u^{*\prime}v=0$,
which force $v_1=0$ and $\sum_{i\ge2}v_i=0$. The stabiliser
of $u^*$, the symmetric group on coordinates $2,\dots,n$,
acts on this $(n-2)$-dimensional tangent space by permuting
coordinates. This is the standard representation, which is
absolutely irreducible: a symmetric matrix on
$\{w\in\R^{n-1}:\sum w_i=0\}$ that commutes with all
coordinate permutations is a multiple of the identity. By
Schur's lemma the Hessian is therefore a multiple of the
identity.
To find the multiple, take the great circle
$u(t)=u^*\cos t+v\sin t$ with $v=2^{-1/2}(0,1,-1,0,\dots,0)'$.
Writing $y=-\{n(n-1)\}^{-1/2}$ and $P_r=p_r(u^*)$,
\[
p_3(t)=P_3\cos^3t+3y\cos t\sin^2t,\qquad
p_4(t)=P_4\cos^4t+6y^2\cos^2t\sin^2t+\tfrac12\sin^4t,
\]
so $p_3''(0)=6y-3P_3$ and $p_4''(0)=12y^2-4P_4$. As $u^*$ is
a critical point, \eqref{eq:JBp} gives
\[
\frac{d^2}{dt^2}\JB_n\{u(t)\}\Big|_{t=0}
=\frac{n^2}{3}P_3\,p_3''(0)+\frac{n^2}{12}(nP_4-3)\,p_4''(0).
\]
Substituting $P_3=(n-2)\{n(n-1)\}^{-1/2}$ and $nP_4=B_n$
and simplifying gives $-\lambda_n$. Finally
$g(n)=n^3-6n^2+15n-12$ has $g(2)=2$ and
$g'(n)=3\{(n-2)^2+1\}>0$, so $\lambda_n>0$ for $n\ge3$.

\emph{Tail.} As in \citet[Sections~2--3]{Mulholland1970}, by the
Morse lemma, near each maximiser there
are local coordinates $\eta\in\R^{n-2}$ with $\eta=0$ at the
maximiser, identity differential there, and
$\JB_n=J_n^+-\frac12\lambda_n|\eta|^2$. In these
coordinates the uniform probability on $\sph$ has a
continuous density $\rho$ with $\rho(0)=1/|\mathbb
S^{n-2}|$, where $|\mathbb S^{n-2}|=2\pi^{(n-1)/2}/
\Gamma\{(n-1)/2\}$. Each maximiser therefore contributes
$\int_{|\eta|^2<2\varepsilon/\lambda_n}\rho(\eta)\,d\eta$,
which is $\rho(0)\,\omega_{n-2}(2\varepsilon/\lambda_n)
^{(n-2)/2}\{1+o(1)\}$, with $\omega_d=\pi^{d/2}/
\Gamma(d/2+1)$ the volume of the unit ball in $\R^d$.
Summing over the $2n$ maximisers gives \eqref{eq:tail}.
Differentiating the same integral in $\varepsilon$ in polar
coordinates, and using the continuity of $\rho$, gives
\eqref{eq:tailden}.
\end{proof}

For $n=3$, \eqref{eq:tailden} gives $(2/\pi)(J_3^+-x)^{-1/2}$,
the arcsine endpoint behaviour of Proposition~\ref{prop:n3}.
For $n=4$ the exponent is zero: the density of $\JB_4$ jumps
from $27/64$ to $0$ at $J_4^+=26/27$. For $n=7$, the leading
term of \eqref{eq:tail}, evaluated at the asymptotic $5\%$
critical value, gives about $1.1\times10^{-5}$, close to the
simulated rejection frequency. The tail
constants agree with simulation for $n=4,\dots,7$.

\section{A discriminant representation}\label{sec:disc}

Let $u_1,\dots,u_n$ be the roots of the monic polynomial
$P(z)=\prod_i(z-u_i)=\sum_{j=0}^n(-1)^je_jz^{n-j}$, where the
$e_j$ are the elementary symmetric polynomials ($e_0=1$).
Newton's identities
$p_r=\sum_{j=1}^{r-1}(-1)^{j-1}e_jp_{r-j}+(-1)^{r-1}re_r$
express $e_1,\dots,e_n$ polynomially in $p_1,\dots,p_n$, so
$P$ is determined by $p=(p_1,\dots,p_n)$. Let
$\mathcal P_n\subset\R^n$ be the set of $p$ for which $P$ has
$n$ real roots, and write
$\Disc(P)=\prod_{i<j}(u_i-u_j)^2$.

\begin{lemma}\label{lem:Zn}
The uniform probability measure on $\sph$ is
$Z_n^{-1}\delta(p_1(u))\,\delta(p_2(u)-1)\,du$, where
\begin{equation}\label{eq:Zn}
Z_n=\frac{\pi^{(n-1)/2}}{\sqrt n\,\Gamma\{(n-1)/2\}} .
\end{equation}
\end{lemma}

\begin{proof}
On $\sph$ the gradients $\nabla p_1=\mathbf 1$ and
$\nabla p_2=2u$ are orthogonal with norms $\sqrt n$ and $2$.
By the coarea formula \citep[3.2.12]{Federer1969}
$\delta(p_1)\delta(p_2-1)\,du$ is surface measure on $\sph$
divided by $2\sqrt n$, and the area of the unit
$(n-2)$-sphere is $2\pi^{(n-1)/2}/\Gamma\{(n-1)/2\}$.
\end{proof}

\begin{theorem}\label{thm:disc}
Let $n\ge 4$. For almost every $(a,b)$ the density of
$(p_3(U),p_4(U))$ is
\begin{equation}\label{eq:fn}
f_n(a,b)=\frac1{Z_n}\int_{\R^{n-4}}
\frac{\mathbf 1\{(0,1,a,b,p_5,\dots,p_n)\in\mathcal P_n\}}
{\Disc(P)^{1/2}}\,dp_5\cdots dp_n ,
\end{equation}
where for $n=4$ the integral is read as evaluation of the
integrand. The density of $(S_n,K_n)$ is
$h_n(s,k)=n^{-3/2}f_n(s/\sqrt n,k/n)$.
\end{theorem}

\begin{proof}
The Jacobian of $u\mapsto p$ has rows
$(ru_1^{r-1},\dots,ru_n^{r-1})$, so its absolute
determinant is $n!\,|\Delta(u)|$ with
$\Delta(u)=\prod_{i<j}(u_j-u_i)$ the Vandermonde determinant.
On the open Weyl chamber $u_1<\dots<u_n$ the map is a
diffeomorphism onto the interior of $\mathcal P_n$, and
$\R^n$ is covered, up to a null set, by the $n!$ images of
the chamber under coordinate permutations, which leave $p$
unchanged. Hence for integrable $G$,
\[
\int_{\R^n}G(p(u))\,du
=n!\int_{u_1<\dots<u_n}G(p(u))\,du
=n!\int_{\mathcal P_n}\frac{G(p)\,dp}{n!\,|\Delta(u)|}
=\int_{\mathcal P_n}\frac{G(p)\,dp}{\Disc(P)^{1/2}}.
\]
Applying this with the measure of
Lemma~\ref{lem:Zn}, whose delta functions now act on the
coordinates $p_1,p_2$, gives \eqref{eq:fn}. The scaling
$s=\sqrt n\,a$, $k=nb$ gives $h_n$.
\end{proof}

The fibre of the moment map over $(a,b)$ has dimension $n-4$.
The following result, which holds for every $n\ge 5$, shows
that the innermost integration in \eqref{eq:fn} is always over
a single interval and is always of hypergeometric type. Fix
$q=(a,b,p_5,\dots,p_{n-1})$, let $e_1,\dots,e_{n-1}$ be
determined by $(0,1,q)$, and put
\[
R_q(z)=\sum_{j=0}^{n-1}(-1)^je_jz^{n-j},\qquad
\gamma(q)=-\frac1n\sum_{j=1}^{n-1}(-1)^{j-1}e_jp_{n-j}.
\]
By Newton's identity for $p_n$,
$P(z)=R_q(z)-t$ with $t=p_n/n+\gamma(q)$: the last power sum
only shifts the graph of $R_q$ vertically. That adding a
constant to a hyperbolic polynomial (one with only real roots)
eventually destroys hyperbolicity, and that it cannot then be
restored, is the simplest case of \citet[Section~2]{Arnold1986};
the next result makes this explicit for the fibre integral.

Fibres fall into three classes: interior ones, where $R_q'$
has $n-1$ simple real zeros and $r_-<r_+$ below; exterior
ones, with fewer real critical points or $r_->r_+$; and
degenerate ones, with repeated critical points or $r_-=r_+$.
The last class is a null set, irrelevant to \eqref{eq:fn}
though not to singularities, so we may assume simple critical
points.

\begin{theorem}[Innermost fibre]\label{thm:fibre}
Let $n\ge 5$ and suppose $R_q'$ has $n-1$ simple real zeros
$w_1<\dots<w_{n-1}$, with critical values $v_j=R_q(w_j)$.
Let $L$ be the largest $v_j$ at a local minimum of $R_q$ and
$U$ the smallest $v_j$ at a local maximum, and set
$r_j=n\{v_j-\gamma(q)\}$, $r_-=n\{L-\gamma(q)\}$,
$r_+=n\{U-\gamma(q)\}$. Then:
\begin{enumerate}
\item[(i)] $(0,1,q,p_n)\in\mathcal P_n$ if and only if
$p_n\in[r_-,r_+]$; this interval is empty if $L>U$, and the
fibre is empty if $R_q'$ has fewer than $n-1$ real zeros;
\item[(ii)] $|\Disc(P)|=n\prod_{j=1}^{n-1}|p_n-r_j|$;
\item[(iii)] if $r_-<r_+$, then, writing $I$ for the set of
the $n-3$ indices $j$ with $r_j\notin\{r_-,r_+\}$ and
$y_j=(r_+-r_-)/(r_j-r_-)$ for $j\in I$,
\begin{equation}\label{eq:FD}
\int_{r_-}^{r_+}\frac{dp_n}{\Disc(P)^{1/2}}
=\frac{\pi\,F_D^{(n-3)}\bigl(\tfrac12;\tfrac12,\dots,
\tfrac12;1;y\bigr)}{\sqrt n\,\prod_{j\in I}|r_--r_j|^{1/2}} .
\end{equation}
\end{enumerate}
\end{theorem}

\begin{proof}
(i) If $R_q'$ has fewer than $n-1$ real zeros, $P=R_q-t$ has
fewer than $n$ real roots by Rolle's theorem. Otherwise the
critical points split $\R$ into $n$ intervals on which $R_q$
is strictly monotone, each containing at most one root of
$R_q-t$; there are $n$ real roots if and only if each interval
contains one, that is, if and only if $t$ lies below every
local maximum value and above every local minimum value of
$R_q$ (closed conditions give the multiple-root boundary).
(ii) For monic $P$ with $P'(z)=n\prod_j(z-w_j)$,
$\Disc(P)=(-1)^{n(n-1)/2}n^n\prod_jP(w_j)$ and
$P(w_j)=v_j-t=(r_j-p_n)/n$.
(iii) Substitute $p_n=r_-+(r_+-r_-)x$; then
$(p_n-r_-)(r_+-p_n)=(r_+-r_-)^2x(1-x)$ and
$|p_n-r_j|=|r_--r_j|(1-y_jx)$ with $y_j<1$, and use Euler's
integral $F_D^{(m)}(a;b;c;y)=\{\Gamma(c)/\Gamma(a)
\Gamma(c-a)\}\int_0^1x^{a-1}(1-x)^{c-a-1}\prod_j
(1-y_jx)^{-b_j}\,dx$ \citep{Exton1976} with $a=\frac12$,
$c=1$.
\end{proof}

\begin{remark}\label{rem:genus}
The right-hand side of \eqref{eq:FD} is a period of the
hyperelliptic curve $w^2=\prod_{j=1}^{n-1}(c-r_j)$, of genus
$\lfloor(n-2)/2\rfloor$. For $n=5$ the genus is one and
$F_D^{(2)}$, an Appell $F_1$, reduces to a complete elliptic
integral (Theorem~\ref{thm:n5}); for $n=6,7$ the genus is two
and no such reduction is available in general. The outer
integrals over $(p_5,\dots,p_{n-1})$ run over the
semialgebraic set on which the conditions of
Theorem~\ref{thm:fibre} hold, whose geometry is that of the
Vandermonde mapping
\citep{Arnold1986,Givental1987,Kostov1989}.
The discriminant representation ties singular behaviour of
$h_n$ to root collisions, that is, to coincident residual
coordinates. A simpler instance of the
same low-dimensional phenomenon occurs for the push-forward of
uniform spherical measure under a quadratic form, the ``real
numerical shadow'', whose density also involves complete
elliptic integrals, equivalently ${}_2F_1(\frac12,\frac12;1;
\cdot)$, in dimensions three to five \citep{Dunkl2015}.
\end{remark}

\section{Exact laws for \texorpdfstring{$n=3,4,5$}{n=3,4,5}}
\label{sec:small}

\subsection{\texorpdfstring{$n=3$}{n=3}}

\begin{proposition}\label{prop:n3}
For $n=3$, $K_3=3/2$ almost surely, $S_3=\cos(3\theta)/
\sqrt2$ with $\theta$ uniform on $[0,2\pi)$, and
$\JB_3=9/32+Y/4$ with $Y\sim\mathrm{Beta}(\frac12,\frac12)$.
Hence $\JB_3$ has support $[9/32,17/32]$, distribution
function $F(x)=(2/\pi)\arcsin\{4(x-9/32)\}^{1/2}$, density
$4\pi^{-1}\{(4x-\frac98)(\frac{17}8-4x)\}^{-1/2}$, and
$p$-quantile $9/32+\frac14\sin^2(\pi p/2)$.
\end{proposition}

\begin{proof}
The unit-speed parametrisation of \citet[Section~3]{Fisher1930},
$u(\theta)=(2/3)^{1/2}\{\cos\theta,\cos(\theta-2\pi/3),
\cos(\theta+2\pi/3)\}$ of $\Omega_3$ gives
$p_3=\cos(3\theta)/\sqrt6$ and $p_4=1/2$, by the identities
$\sum_j\cos^3(\theta+2\pi j/3)=\frac34\cos3\theta$ and
$\sum_j\cos^4(\theta+2\pi j/3)=\frac98$. Apply
Proposition~\ref{prop:power}; $\cos^2$ of a uniform angle is
$\mathrm{Beta}(\frac12,\frac12)$.
\end{proof}

The density of $S_3$ is therefore
$\sqrt2\,\pi^{-1}(1-2s^2)^{-1/2}$, $|s|<2^{-1/2}$, with
$K_3$ degenerate, as in \citet{Fisher1930}; this is also the
starting point of NHN's recurrence.

\subsection{\texorpdfstring{$n=4$}{n=4}}\label{sec:n4}

For $n=4$, $(p_1,p_2,p_3,p_4)=(0,1,a,b)$ gives
$P_4(z)=z^4-\frac12z^2-\frac a3z+\frac{1-2b}8$, and with
$a=s/2$, $b=k/4$,
\begin{equation}\label{eq:Q4}
144\,\Disc(P_4)=Q_4(s,k)=18-45k+36k^2-9k^3-34s^2
+18s^2k-3s^4 .
\end{equation}

\begin{theorem}\label{thm:n4}
The support of $(S_4,K_4)$ is the closure of
$\mathcal D_4=\{(s,k):Q_4(s,k)>0,\ k>1\}$, and
\begin{equation}\label{eq:h4}
h_4(s,k)=\frac{3}{2\pi\,Q_4(s,k)^{1/2}},\qquad
(s,k)\in\mathcal D_4 .
\end{equation}
\end{theorem}

\begin{proof}
For $n=4$, $Z_4=\pi$ and Theorem~\ref{thm:disc} has no fibre
variables, so $h_4=\{8\pi\,\Disc(P_4)^{1/2}\}^{-1}$ wherever
$P_4$ has four real roots; \eqref{eq:Q4} gives \eqref{eq:h4}.
It remains to identify where $P_4$ has four real roots. If
$Q_4<0$, $P_4$ has exactly two. If $Q_4>0$, its roots are
either all real or two non-real conjugate pairs. In the
second case write them $\alpha\pm i\beta$, $\gamma\pm i
\delta$ with $\beta,\delta>0$; since $e_1=0$, $\gamma=-\alpha$,
and with $\sigma=\beta^2+\delta^2>0$ the condition $p_2=1$
gives $\alpha^2=(1+2\sigma)/4$. Then
\[
p_4=4\alpha^4-12\alpha^2\sigma+2(\beta^4+\delta^4)
\le 4\alpha^4-12\alpha^2\sigma+2\sigma^2
=\tfrac14-2\sigma-3\sigma^2<\tfrac14,
\]
that is, $k=4p_4<1$. Conversely, four real roots give
$k\ge 1$ by Lemma~\ref{lem:bounds}.
\end{proof}

\begin{remark}\label{rem:triangle}
The curve $Q_4=0$ has exactly three singular points: a node
at $(0,1)$, the image of the two-pair pattern $(x,x,-x,-x)$,
where the support has a corner with tangents
$k-1=\pm\frac43s$; and cusps at $(\pm2/\sqrt3,7/3)$, the
images of the single-outlier pattern. The support is the
curvilinear triangle with these vertices
(Figure~\ref{fig:support}a); its edges are images of
configurations with exactly one coincident pair. On the
axis, $Q_4(0,k)=9(k-1)^2(2-k)$, so the support meets $s=0$
in $[1,2]$. In the coefficients $(\lambda_3,\lambda_4)=
(-a/3,(1-2b)/8)$ of $P_4$, the support is the section
$\lambda_2=-\frac12$ of the set of hyperbolic quartics
$z^4+\lambda_2z^2+\lambda_3z+\lambda_4$, which
\citet{Arnold1986} describes as a pyramid bounded by a
swallowtail; the node and the cusps correspond to the
self-intersection and the cuspidal edges of the swallowtail.
\citet{DeBartolo2026} derive this boundary analytically as
the curve $(s,k)=(3r^3-3r,\,2+2r^2-3r^4)$, $|r|\le1$, which
satisfies $Q_4=0$ identically. It passes through $(0,1)$ twice,
at $r=\pm1$, with the two tangents above, so the lower vertex
is a corner rather than a cusp; the cusps are at
$r=\pm3^{-1/2}$. Here the boundary arises as the discriminant
locus of the quartic $P_4$, and Theorem~\ref{thm:n4} also gives
the density on its interior. The support is also strictly
smaller than the region allowed by the classical bounds: for
$n=4$ the bound of \citet{SharmaBhandari2015} reads
$K\le 2+S^2/4$, and $Q_4(s,2+s^2/4)=-s^2(3s^2-4)^2/64$, so
this parabola touches the support only at $(0,2)$ and at the
two cusps.
\end{remark}

\begin{figure}[tbp]
\centering
\includegraphics[width=\textwidth]{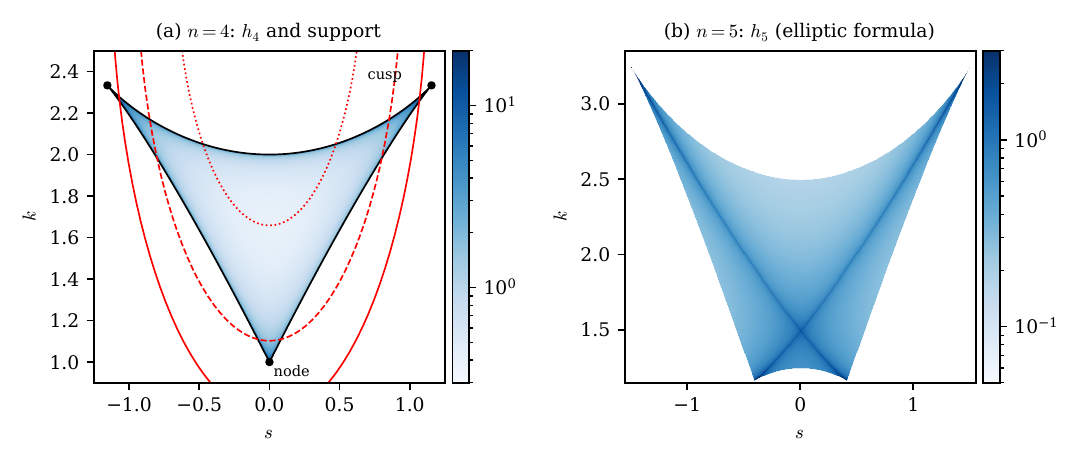}
\caption{(a) Exact joint density \eqref{eq:h4} of
$(S_4,K_4)$ (log colour scale) on its curvilinear-triangle
support, with the node and cusps marked and the ellipses
$\JB_4=0.3$, $0.6$ and $0.852$ (the $95\%$ quantile).
(b) Exact joint density of $(S_5,K_5)$ from
Theorem~\ref{thm:n5}; the dark curves are the interior
logarithmic singularities.}
\label{fig:support}
\end{figure}

Theorem~\ref{thm:n4} sharpens an observation of
\citet{ShentonBowman1977}, whose simulations showed a
pronounced two-lobed shape in the conditional distribution of
$S_n$ given large $K_n$. For fixed $k$, $Q_4$ is a quadratic
in $s^2$. If $1<k<2$, the support slice is a single interval
symmetric about zero, and the conditional density,
proportional to $Q_4(s,k)^{-1/2}$, has inverse-square-root
singularities at its two endpoints. If $2<k<7/3$, the slice
splits into two symmetric intervals excluding zero, with such
singularities at all four endpoints.

For fixed $s$, $Q_4$ is a cubic in $k$, which makes the
marginal of $S_4$ and the distribution function of $\JB_4$
elliptic. The marginal of $S_4$ recovers \citet{McKay1933}.

\begin{corollary}\label{cor:S4}
Let $0<|s|<2/\sqrt3$. Then $Q_4(s,\cdot)$ has three real
zeros $k_1<1<k_2<k_3$, the support slice is $(k_2,k_3)$, and
the density of $S_4$ is
\begin{equation}\label{eq:gS4}
g_4(s)=\frac{\EK(m)}{\pi(k_3-k_1)^{1/2}}
=\frac{1}{2\sqrt3}\,{}_2F_1\Bigl(\frac13,\frac23;1;
1-\frac{3s^2}{4}\Bigr),\qquad
m=\frac{k_3-k_2}{k_3-k_1},
\end{equation}
where $\EK(m)=\int_0^{\pi/2}(1-m\sin^2\phi)^{-1/2}d\phi
=\frac\pi2\,{}_2F_1(\frac12,\frac12;1;m)$.
\end{corollary}

\begin{proof}
The discriminant of $Q_4(s,\cdot)$ is
$729\,s^2(4-3s^2)^3>0$, so the zeros are real and distinct.
Since $Q_4(s,1)=-s^2(16+3s^2)<0$ and $Q_4\to+\infty$ as
$k\to-\infty$, $1\in(k_1,k_2)$ or $1>k_3$. In the second case
the slice would be empty near $s$, which is impossible
because $S_4$ is continuous on the connected set
$\Omega_4$ with range $[-2/\sqrt3,2/\sqrt3]$. Integrate
\eqref{eq:h4} over $(k_2,k_3)$ using
$Q_4=9(k-k_1)(k-k_2)(k_3-k)$ and \citet[236.00]{ByrdFriedman1971}
to obtain the first expression. The second is the closed form
of \citet[eq.~(36)]{McKay1933}, who obtained the density of
$S_4$ as a complete elliptic integral and reduced it, via
Mehler's integral, to a Legendre function. Both expressions
are the density of $S_4$ and so coincide; numerically they
agree to about $10^{-13}$, and their equality amounts to a
transformation between the two hypergeometric functions.
\end{proof}

McKay's form shows that the density of $S_4$ equals
$1/(2\sqrt3)$ at the ends $\pm2/\sqrt3$ of its range and has
a logarithmic singularity at $0$, the image of the node.

\begin{proposition}\label{prop:JB4}
$\JB_4$ has support $[1/6,26/27]$. The critical values of
$\JB_4$ restricted to the boundary of the support are
$1/6$ (at $(0,2)$), $x_*=(145\sqrt{145}-1729)/36\approx
0.47309$ (tangency with the lower edges, where
$k=\frac32(\sqrt{145}-11)$), $2/3$ (at the node) and $26/27$
(at the cusps). For $x\ge 0$,
\begin{equation}\label{eq:FJB4}
F_{\JB_4}(x)=\frac1\pi\int
\frac{\EK(m)-\EF(\phi_x\mid m)}{(k_3-k_1)^{1/2}}\,ds,
\qquad
\sin^2\phi_x=\frac{(k_3-k_1)(y-k_2)}{(k_3-k_2)(y-k_1)},
\end{equation}
where $k_j$ and $m$ are as in Corollary~\ref{cor:S4},
$y=\max\{k_2,3-(6x-4s^2)^{1/2}\}$, $\EF(\cdot\mid m)$ is the
incomplete elliptic integral of the first kind, and the
integral runs over those $s$ with $4s^2<6x$ and $y<k_3$.
\end{proposition}

\begin{proof}
$\JB_4=\frac23s^2+\frac16(k-3)^2$ has its only unconstrained
critical point at $(0,3)\notin\overline{\mathcal D_4}$, so
its extrema and its critical values on the support lie on
$\{Q_4=0\}$. Eliminating $s$ from $Q_4=0$ and the Lagrange
condition gives $(k-2)(k-1)^2(3k-7)^3(k^2+33k-54)=0$; the
admissible points give the four listed values, and the
extreme ones are $J_4^-=1/6$ and $J_4^+=26/27$. For
\eqref{eq:FJB4}, $\JB_4\le x$ if and only if
$|k-3|\le(6x-4s^2)^{1/2}$; since $k_3\le 7/3<3$, the slice is
$(y,k_3)$, and \citet[236.00]{ByrdFriedman1971} gives the
inner integral.
\end{proof}

\begin{figure}[tbp]
\centering
\includegraphics[width=\textwidth]{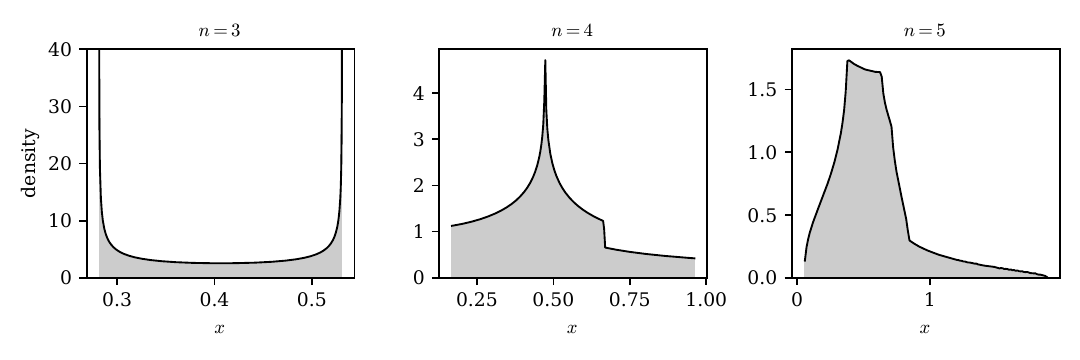}
\caption{Exact densities of $\JB_n$ (solid) and histograms
of $10^6$ ($n=3$) or $10^7$ ($n=4,5$) simulated values. For
$n=4$ the density is obtained by differentiating
\eqref{eq:FJB4}; note the logarithmic singularity at
$x_*\approx0.473$ and the jump at $2/3$. For $n=5$ it is
computed from Theorem~\ref{thm:n5}.}
\label{fig:JB}
\end{figure}

The critical values in Proposition~\ref{prop:JB4} are exactly
where the density of $\JB_4$ is not smooth, and its singular
behaviour there can be found exactly
(Figure~\ref{fig:JB}).

\begin{proposition}\label{prop:JB4sing}
On $\Omega_4$, $\JB_4$ has exactly $50$ stationary points,
all nondegenerate: $12$ minima at the permutations of
$2^{-1/2}(0,0,1,-1)'$, where $\JB_4=1/6$ and the Hessian has
eigenvalues $8/3$ and $32/3$; $24$ saddle points, where
$\JB_4=x_*$ and the Hessian has determinant
$\det H_*=(5286656\sqrt{145}-63660800)/27\approx-38.2137$;
$6$ local maxima at the
permutations of $\frac12(1,1,-1,-1)'$, where $\JB_4=2/3$ and
the Hessian is $-\frac{16}3$ times the identity; and the $8$
global maxima of Proposition~\ref{prop:tail}. Consequently
$f_{\JB_4}$ is infinitely differentiable on $(1/6,26/27)$
except at $x_*$ and $2/3$, and
\[
f_{\JB_4}(\tfrac16+)=\tfrac98,\qquad
f_{\JB_4}(\tfrac23-)-f_{\JB_4}(\tfrac23+)=\tfrac9{16},\qquad
f_{\JB_4}(\tfrac{26}{27}-)=\tfrac{27}{64},
\]
\[
f_{\JB_4}(x)=-\frac{12}{\pi\,|\det H_*|^{1/2}}\,
\log|x-x_*|+O(1)\approx-0.618\log|x-x_*|
\qquad(x\to x_*).
\]
\end{proposition}

\begin{proof}
Write $\JB_4=J\circ\Phi$ with $\Phi(u)=(S_4,K_4)$ and
$J(s,k)=\frac23s^2+\frac16(k-3)^2$. Where the coordinates of
$u$ are distinct, $d\Phi_u$ maps the tangent space onto $\R^2$
(the Jacobian of $(p_1,\dots,p_4)$ is a Vandermonde
determinant), so $u$ can be stationary only if
$\nabla J\{\Phi(u)\}=0$, that is, $\Phi(u)=(0,3)$, which is
outside the support. Stationary points therefore have
coincident coordinates and lie over the curve $Q_4=0$. Over a
smooth point of the curve, $d\Phi_u$ has rank one with image
tangent to the curve, so $u$ is stationary if and only if
$\Phi(u)$ is a critical point of $J$ on the curve; by the
proof of Proposition~\ref{prop:JB4} this gives the values
$1/6$ and $x_*$. Over the node and the cusps $d\Phi_u=0$;
their preimages are the two-pair and single-outlier
directions. At a saddle, write the configuration as
$(x,x,y,z)$ up to permutation. Stationarity gives
$12x^4+7x^2-2=0$, so $x^2=(\sqrt{145}-7)/24$, and the
constraints $2x+y+z=0$, $2x^2+y^2+z^2=1$ give
$y,z=-x\pm\{(1-4x^2)/2\}^{1/2}$. Two signs of $x$, six
positions for the repeated pair and two assignments of $y,z$
give $24$ saddles; the other counts are $4!/2!=12$,
$4!/(2!\,2!)=6$ and $2\times4=8$. Differentiation on the sphere
gives the stated Hessians; $\det H_*$ is the negative root of
$27d^2+127321600d+4865392640=0$. As a check,
$12-24+6+8=2$ is the Euler characteristic of the
$2$-sphere. By the local analysis of
\citet[Sections~2--3]{Mulholland1970}, equivalently the Morse-lemma
argument in the proof of Proposition~\ref{prop:tail} with
$\Omega_4$ of dimension two and uniform density $1/(4\pi)$, a
nondegenerate local extremum with Hessian $H$ contributes a
step of height $1/(2|\det H|^{1/2})$ to the density at its
value, a saddle contributes
$-\{2\pi|\det H|^{1/2}\}^{-1}\log|x-x_*|$, and the remainder
is continuous. Summing over the stationary points gives the
stated values.
\end{proof}

\subsection{\texorpdfstring{$n=5$}{n=5}}\label{sec:n5}

For $n=5$ put $c=p_5$. Newton's identities give
$e_5=c/5-a/6$ and
\begin{equation}\label{eq:P5}
P_5(z)=R(z)-t,\quad
R(z)=z^5-\tfrac12z^3-\tfrac a3z^2+\tfrac{1-2b}8z,\quad
t=\tfrac c5-\tfrac a6 .
\end{equation}
$\Disc(P_5)$ is a quartic in $c$ with leading coefficient $5$;
it is displayed in Appendix~\ref{app:disc}.

\begin{theorem}\label{thm:n5}
Let $(a,b)$ be such that $R'$ has four simple real zeros
$w_1<w_2<w_3<w_4$ (local maximum, minimum, maximum, minimum
of $R$), with values $v_j=R(w_j)$ satisfying
$\max(v_2,v_4)<\min(v_1,v_3)$. Let $r_1<r_2<r_3<r_4$ be the
numbers $5(v_j+a/6)$ in increasing order. Then $(a,b)$ is
an interior point of the support of $(p_3,p_4)$,
$\Disc(P_5)=5\prod_j(c-r_j)$, the fibre is $c\in[r_2,r_3]$,
and
\begin{equation}\label{eq:f5}
f_5(a,b)=\frac{2\,\EK(m)}
{\pi^2\{(r_3-r_1)(r_4-r_2)\}^{1/2}},\qquad
m=\frac{(r_3-r_2)(r_4-r_1)}{(r_3-r_1)(r_4-r_2)} .
\end{equation}
If $R'$ has fewer than four real zeros, or
$\max(v_2,v_4)>\min(v_1,v_3)$, then $(a,b)$ lies outside the
support and $f_5(a,b)=0$; equality describes the support
boundary, a null set. The density of $(S_5,K_5)$ is
$h_5(s,k)=5^{-3/2}f_5(s/\sqrt5,k/5)$.
\end{theorem}

\begin{proof}
This is Theorem~\ref{thm:fibre} with $n=5$ and
$\gamma=-a/6$: the two local-minimum values give $r_1,r_2$
and the two local-maximum values give $r_3,r_4$, so
$r_-=r_2$ and $r_+=r_3$. By Lemma~\ref{lem:Zn},
$Z_5=\pi^2/\sqrt5$, and \eqref{eq:fn} becomes
$f_5=(\sqrt5/\pi^2)\int_{r_2}^{r_3}\{5\prod_j(c-r_j)\}^{-1/2}
\,dc$. The classical reduction
\citep[3.147.4]{GradshteynRyzhik2007} gives
$\int_{r_2}^{r_3}\{\prod_j|c-r_j|\}^{-1/2}dc
=2\EK(m)\{(r_3-r_1)(r_4-r_2)\}^{-1/2}$.
\end{proof}

On the interior, \eqref{eq:f5} is the density; no further
case distinction is needed. Its form explains the
singularities of $h_5$ (Figure~\ref{fig:support}b). At a
regular boundary point $r_2\uparrow r_3$, so $m\to0$ and $h_5$
has a jump to a positive
limit; this contrasts with the inverse-square-root boundary
behaviour for $n=3,4$. In the interior, $h_5$ has
logarithmic singularities along the curves $v_2=v_4$ and
$v_1=v_3$, where $m\to1$; there two pairs of residuals
coincide. In particular $h_5$ is bounded near every smooth
boundary point, and the loss of integrable singularities at
the boundary is the smoothing effect of the one-dimensional
fibre integral.

\section{The distribution of \texorpdfstring{$\JB_n$}{JBn}}
\label{sec:JB}

\begin{proposition}\label{prop:polar}
If $(S_n,K_n)$ has density $h_n$, then at continuity points
$x>0$,
\begin{equation}\label{eq:polar}
f_{\JB_n}(x)=\frac6n\int_0^{2\pi}h_n\Bigl(
\bigl(\tfrac{6x}{n}\bigr)^{1/2}\cos\theta,\;
3+2\bigl(\tfrac{6x}{n}\bigr)^{1/2}\sin\theta\Bigr)d\theta .
\end{equation}
\end{proposition}

\begin{proof}
Put $s=(6/n)^{1/2}r\cos\theta$ and
$k-3=2(6/n)^{1/2}r\sin\theta$, so that $\JB_n=r^2$ and
$ds\,dk=(12r/n)\,dr\,d\theta=(6/n)\,dx\,d\theta$ with
$x=r^2$.
\end{proof}

Table~\ref{tab:quant} reports the range and upper quantiles
of $\JB_n$. For $n=3$ they follow from
Proposition~\ref{prop:n3}; for $n=4$ from \eqref{eq:FJB4},
evaluated in $25$-digit arithmetic, which gives
$0.7552955$, $0.8519539$ and $0.9395652$. For $n=5$,
quadrature of Theorem~\ref{thm:n5} over a $3200^2$ grid in
$(a,b)$ gives $0.94422$, $1.21846$ and $1.62753$, and $10^8$
Monte Carlo replications give $0.94414$, $1.21850$ and
$1.62731$; the two agree to within $3\times10^{-4}$, and
Table~\ref{tab:quant} reports three decimals.
For $n\ge6$ the entries are simulated. The simulated values
at $n=30$ and $75$ agree to within $0.01$ with the
response-surface values of \citet{Lawford2005}.

\begin{table}[tbp]
\centering
\caption{Range and upper quantiles of $\JB_n$ under
normality, and size of the asymptotic $5\%$ test. The range
endpoints $J_n^\pm$ are exact (Theorem~\ref{thm:range}). For
$n\le4$ the quantiles are computed from the exact distribution
functions of Section~\ref{sec:small}; for $n=5$ they are given
to the three decimals on which quadrature of the exact density
of Theorem~\ref{thm:n5} and simulation agree (see text). Rows
$n\ge6$ are based on $2\times10^7$ ($10^8$ for $n=7$; $10^7$ for
$n=75$) replications; the size is zero for $n\le6$ by
Corollary~\ref{cor:size}. Asymptotic ($\chi^2_2$) quantiles:
$4.605$, $5.991$, $9.210$.}
\label{tab:quant}
\small
\input{table1.tex}
\end{table}

\paragraph{Checks against exact moments.}
Table~\ref{tab:mom} compares product moments computed from
\eqref{eq:h4} and \eqref{eq:f5} with exact rational values.
The exact values need no distribution theory. Write
$Y=H'W$, where $H$ is the $(n-1)\times n$ Helmert contrast
matrix and $W\sim N(0,I_{n-1})$. By Lemma~\ref{lem:sphere},
for a homogeneous polynomial $q$ of degree $d$,
$\E q(U)=\E q(Y)/\E\|Y\|^d$, where
$\E\|Y\|^d=2^{d/2}\Gamma\{(n-1+d)/2\}/\Gamma\{(n-1)/2\}$.
The numerator is a finite sum of Gaussian moments of $W$.
The resulting values coincide with those given by NHN's
recurrence and, for $\E S_n^2$ and the first two moments of
$K_n$, with the classical formulae going back to
\citet{Fisher1930}.
For $n=4$ the integrals were evaluated by adaptive
quadrature after the substitution
$k=k_2+(k_3-k_2)\sin^2\phi$, which removes the boundary
singularity; the agreement is to about $10^{-12}$.

\begin{table}[tbp]
\centering
\caption{Product moments of $(S_n,K_n)$: exact values by
spherical integration (see text) and values computed from the
densities of Theorems~\ref{thm:n4} and~\ref{thm:n5}. For $n=5$,
midpoint quadrature on a $3200^2$ grid; its error, about
$10^{-4}$, reflects the logarithmic singular curves of
Figure~\ref{fig:support}b.}
\label{tab:mom}
\small
\input{table2.tex}
\end{table}

\FloatBarrier
\section{Discussion and open problems}\label{sec:open}

NHN solve a recursion problem, expressing the $n$-sample
density and moments through the $(n-1)$-sample objects; the
present paper solves, for small $n$, a fibre-integration
problem. The recurrence is suited to symbolic moments and
induction in $n$; the discriminant representation displays
support, singularities and special-function structure.
The sequence $n=3,4,5$ has a coherent structure: an arcsine
law, an algebraic joint density whose skewness marginal is
McKay's Gauss hypergeometric function, and a joint density
that is a single elliptic period. For
$n\ge6$, Theorem~\ref{thm:fibre} still gives a Lauricella
period at the innermost level, but the outer $(n-5)$-fold
integral does not appear to preserve hypergeometric form.
Two problems seem both natural and tractable.

\begin{openproblem}[Deterministic evaluation for moderate
$n$]\label{op:holonomic}
For each fixed $n$, $f_{\JB_n}$ is represented by period
integrals of an algebraic family, suggesting that it should
be holonomic: that it satisfies a linear differential equation
in $x$ with polynomial coefficients. The moments $\E\JB_n^r$
would then typically satisfy a linear recurrence in $r$ with polynomial
coefficients, which can be guessed from exact moments
computed by NHN's recurrence or by the spherical integration
used for Table~\ref{tab:mom}. Such
an equation would allow evaluation of $F_{\JB_n}$ by the
holonomic gradient method \citep{Nakayama2011,Hashiguchi2013},
replacing simulation in the calibration problem of
\citet{DebSefton1996} and \citet{Lawford2005}. Determine the
order of the equation as a function of $n$ and whether it is
ever of hypergeometric (first-order recurrence) type.
\end{openproblem}

\begin{openproblem}[Singularity classification]
\label{op:sing}
Classify the singularities of $f_{\JB_n}$ and of $h_n$. For
$f_{\JB_n}$, the theory of \citet{Mulholland1970} gives the
leading singular term at every nondegenerate stationary point
of $\JB_n$ on $\sph$; Propositions~\ref{prop:tail}
and~\ref{prop:JB4sing} carry this out at the maximum for
every $n$ and completely for $n=4$. For $n\ge6$ the minimum
set $\{S_n=0,K_n=3\}$ has positive dimension, so the
stationary points are no longer isolated, and the remaining
stationary points have not been classified. For the joint
density the classification is open beyond the cases treated
here: for $n=4$ the singularities are inverse-square-root at
the edges of the support, and for $n=5$ the density jumps at
the boundary and is logarithmic along interior curves
(Section~\ref{sec:n5}).
\end{openproblem}

The same spherical mechanism persists for Gaussian regression
residuals, whose direction is uniform on the unit sphere of
the residual subspace. The obstruction is then algebraic
rather than probabilistic: the coordinate power sums depend on
the orientation of that subspace relative to the coordinate
axes. Characterising the designs for which a tractable
discriminant or period representation survives would extend
the present geometry from one-sample normality testing to
regression diagnostics.

\appendix
\section{The \texorpdfstring{$n=5$}{n=5} discriminant}
\label{app:disc}

With $P_5$ as in \eqref{eq:P5},
\begin{align*}
\Disc(P_5)={}&5c^4-\tfrac{35a}{3}c^3
-\tfrac{1}{400}\bigl(1000a^2b-4200a^2+1000b^2-550b+79\bigr)
c^2\\
&+\tfrac{a}{720}\bigl(64a^4+2580a^2b-3234a^2+1200b^3
+1620b^2-1116b+165\bigr)c\\
&-\tfrac{1}{13824}\bigl(1024a^6+288a^4b^2+16992a^4b
-10872a^4+17472a^2b^3+4464a^2b^2\\
&\qquad-5904a^2b+948a^2+3456b^5-6912b^4+5400b^3
-2052b^2+378b-27\bigr).
\end{align*}
This expression, \eqref{eq:Q4} and every other algebraic
identity used in the paper were verified symbolically in
SymPy; the script is part of the supplementary material. By Theorem~\ref{thm:fibre}(ii) its four roots in $c$ are $5(v_j+a/6)$, where the $v_j$ are the critical values of $R$ in \eqref{eq:P5}.

\section{Computation}\label{app:comp}

All densities were checked against simulation from the null
using $U_i=(Z_i-\bar Z)/\{\sum_j(Z_j-\bar Z)^2\}^{1/2}$ with
$Z_i$ independent standard normal. For $n=5$, the zeros of
$R'$ are computed as eigenvalues of the companion matrix and
\eqref{eq:f5} is evaluated with the complete elliptic
integral. Against $10^7$ simulated pairs $(S_5,K_5)$, binned
in $138$ interior boxes of side $0.02$ away from the boundary
and the singular curves, the largest standardised discrepancy
is $2.42$ and $97.8\%$ are below $1.96$ in absolute value.
The exact moments in
Table~\ref{tab:mom} were computed by the spherical
integration described in Section~\ref{sec:JB}. Polynomial
roots were computed as companion-matrix eigenvalues, elliptic
integrals with SciPy and mpmath, and one-dimensional integrals
by adaptive (QUADPACK or tanh--sinh) quadrature. The exact
elimination for the stationary points in
Proposition~\ref{prop:JB4sing} was carried out in SymPy; the
lower endpoint $J_5^-=5/96$ is proved analytically in
Theorem~\ref{thm:range}, and the supplementary code also
checks it independently. Scripts that reproduce every table,
figure and numerical statement, with the software versions
used, are provided as supplementary material.

{\small
\bibliographystyle{plainnat}
\bibliography{ms}
}

\end{document}

%% file: table1.tex
\begin{tabular}{rrrrrrr}
\toprule
$n$ & $J_n^-$ & $J_n^+$ & $90\%$ & $95\%$ & $99\%$
 & size at $5.991$\\
\midrule
3 & $9/32$ & $17/32$ & 0.5251 & 0.5297 & 0.5312 & 0\\
4 & $1/6$ & $26/27$ & 0.7553 & 0.8520 & 0.9396 & 0\\
5 & $5/96$ & $725/384$ & 0.944 & 1.219 & 1.627 & 0\\
\midrule
6 & 0 & 3.560 & 1.099 & 1.553 & 2.482 & 0\\
7 & 0 & 6.230 & 1.252 & 1.842 & 3.372 & $1.2\times10^{-5}$\\
8 & 0 & 10.15 & 1.391 & 2.092 & 4.227 & 0.0024\\
10 & 0 & 22.74 & 1.622 & 2.520 & 5.694 & 0.0088\\
20 & 0 & 245.7 & 2.352 & 3.802 & 9.764 & 0.0244\\
30 & 0 & 918.6 & 2.744 & 4.406 & 11.33 & 0.0310\\
75 & 0 & 16219 & 3.484 & 5.274 & 12.58 & 0.0404\\
\bottomrule
\end{tabular}

%% file: table2.tex
\begin{tabular}{lrrrr}
\toprule
 & \multicolumn{2}{c}{$n=4$} & \multicolumn{2}{c}{$n=5$}\\
\cmidrule(lr){2-3}\cmidrule(lr){4-5}
Moment & Exact & from \eqref{eq:h4} & Exact
 & from \eqref{eq:f5}\\
\midrule
$1$ & $1$ & $1.0000000000$ & $1$ & $1.00001$\\
$\E S^2$ & $12/35$ & $0.3428571429$ & $3/8$ & $0.37498$\\
$\E K$ & $9/5$ & $1.8000000000$ & $2$ & $2.00000$\\
$\E S^4$ & $1296/5005$ & $0.2589410589$ & $81/224$ & $0.36157$\\
$\E K^2$ & $353/105$ & $3.3619047619$ & $17/4$ & $4.24994$\\
$\E S^2K$ & $276/385$ & $0.7168831169$ & $15/16$ & $0.93744$\\
\bottomrule
\end{tabular}